\documentclass[12pt]{article}
\usepackage{latexsym, amsmath, amsfonts, amsthm, amssymb}
\usepackage{times}
\usepackage{a4wide}
\usepackage{times}
\usepackage{graphicx, tocloft}

\def \RR {\mathbb R}
\def \NN {\mathbb N}
\def \EE {\mathbb E}

\def \PP {\mathbb P}

\def \eps {\varepsilon}
\def \vphi {\varphi}

\newtheorem{theorem}{Theorem}[section]
\newtheorem{lemma}[theorem]{Lemma}
\newtheorem{proposition}[theorem]{Proposition}
\newtheorem{corollary}[theorem]{Corollary}
\newtheorem{remark}[theorem]{Remark}

\def\myffrac#1#2 in #3{\raise 2.6pt\hbox{$#3 #1$}\mkern-1.5mu\raise 0.8pt\hbox{$
#3/$}\mkern-1.1mu\lower 1.5pt\hbox{$#3 #2$}}

\def\qed{\hfill $\vcenter{\hrule height .3mm
\hbox {\vrule width .3mm height 2.1mm \kern 2mm \vrule width .3mm
height 2.1mm} \hrule height .3mm}$ \bigskip}

\begin{document}

\title{Logarithmically-concave moment measures I}
\date{}
\author{Bo'az Klartag\thanks{School of Mathematical Sciences, Tel Aviv University, Tel Aviv 69978, Israel. E-mail: klartagb@tau.ac.il. }}
\maketitle

\abstract{We discuss a certain Riemannian metric, related to the toric K\"ahler-Einstein equation,
that is associated in a linearly-invariant manner with a given log-concave measure in $\RR^n$.
We use this metric in order to bound the second derivatives of the solution to the toric K\"ahler-Einstein equation,
and in order to obtain spectral-gap estimates similar to those of Payne and
Weinberger.}

\renewcommand\cftsecfont{\normalsize}
\renewcommand\cftsecpagefont{\normalsize}
\renewcommand{\cftsecleader}{\cftdotfill{\cftdotsep}}
\renewcommand\cftsecafterpnum{\par\vspace{-10pt}}
\tableofcontents

\section{Introduction}
\label{sec1}

In this paper we
 explore a certain geometric structure related to the {\it moment measure} of a convex function.
This geometric structure is well-known in the community of complex geometers, see, e.g., Donaldson \cite{donaldson} for
a discussion from the perspective of K\"ahler geometry.

\medskip Our motivation stems from the Kannan-Lovas\'z-Simonovits conjecture \cite[Section 5]{KLS},
which is concerned with the isoperimetric problem for high-dimensional convex bodies.
Essentially, our idea
is to replace the standard Euclidean metric by a special Riemannian metric on the given convex body $K$. This Riemannian
metric has many favorable properties, such as a Poincar\'e inequality with constant one, a positive Ricci tensor, the linear functions are
eigenfunctions of the Laplacian, etc. Perhaps this alternative geometry does not deviate too much from the standard Euclidean geometry on $K$,
and it is conceivable that the study of this Riemannian metric will turn out to be relevant to the Kannan-Lovas\'z-Simonovits conjecture.

\medskip Let $\mu$ be an arbitrary Borel probability measure on $\RR^n$ whose barycenter is at the origin.
Assume furthermore that $\mu$ is not supported in a hyperplane. It was proven in \cite{CK}
that there
exists an essentially-continuous convex function $\psi: \RR^n \rightarrow \RR \cup \{+\infty \}$, uniquely determined up to translations,  such that
$\mu$ is the {\it moment measure} of $\psi$, i.e.,
$$
\int_{\RR^n} b(y) d \mu(y) = \int_{\RR^n} b(\nabla \psi(x)) e^{-\psi(x)} dx $$
for any $\mu$-integrable function $b: \RR^n \rightarrow \RR$. In other words, the gradient map $x \mapsto \nabla \psi(x)$
pushes the probability measure $e^{-\psi(x)} dx$ forward to $\mu$. The argument in \cite{CK} closely follows
 the variational approach of Berman and Berndtsson \cite{BB}, which succeeded the continuity methods of Wang and Zhu \cite{WZ}
and Donaldson \cite{donaldson}.

\medskip Even in the case where $\mu$ is absolutely-continuous with a $C^{\infty}$-smooth density, it is not guaranteed that $\psi$ is differentiable.
From the regularity theory of the Brenier map, developed by Caffarelli \cite{caf1} and Urbas \cite{urbas}, we
learn that in order to conclude that $\psi$ is sufficiently smooth, one has to assume that the support of $\mu$ is convex.

\medskip An absolutely-continuous probability measure on $\RR^n$ is called {\it log-concave}
if it is supported on an open, convex set $K \subset \RR^n$, and its density takes the form $\exp(-\rho)$
where the function $\rho: K \rightarrow \RR$ is convex. An important example of a log-concave measure
is the uniform probability measure on a convex body in $\RR^n$. Here we assume that
$\mu$ is log-concave and furthermore, we require  that the following conditions are met:
\addtocounter{equation}{1}
\newcounter{eq_935}
\addtocounter{eq_935}{\value{equation}}
\begin{enumerate}
\item[(\arabic{equation})] The convex set $K \subset \RR^n$ is bounded, the function $\rho$
is $C^{\infty}$-smooth, and $\rho$ and its derivatives of all orders are bounded in $K$.
\end{enumerate}
Under these regularity assumptions, we can assert that
\addtocounter{equation}{1}
\newcounter{eq_936}
\addtocounter{eq_936}{\value{equation}}
\begin{enumerate}
\item[(\arabic{equation})] The convex function $\psi$
is finite and $C^{\infty}$-smooth in the entire $\RR^n$.
\end{enumerate}
The validity of (\arabic{eq_936}) under the assumption (\arabic{eq_935}) was proven
by Wang and Zhu \cite{WZ} and by Donaldson \cite{donaldson} via the continuity method.
 Berman and Berndtsson \cite{BB} explained how to deduce (\arabic{eq_936}) from (\arabic{eq_935})
by using Caffarelli's regularity theory \cite{caf1}. In fact, the argument in \cite{BB} requires only the boundness
of $\rho$, and not of its derivatives, see also the Appendix in Alesker, Dar and Milman \cite{ADM}.
Since the function $\psi$ is smooth, the transport equation
\begin{equation} e^{-\rho(\nabla \psi(x))} \det \nabla^2 \psi(x) = e^{-\psi(x)}
\label{eq_1120}
\end{equation}
holds everywhere in $\RR^n$, where
$\nabla^2 \psi(x)$ is the Hessian matrix of $\psi$ (see, e.g., McCann \cite{mccann}). In the case where $\rho \equiv Const$,
equation (\ref{eq_1120}) is called the {\it toric K\"ahler-Einstein equation}.
 We write $x \cdot y$ for the standard
scalar product of $x,y \in \RR^n$, and $|x| = \sqrt{x \cdot x}$.

\begin{theorem} Let $\mu$ be a log-concave probability measure on $\RR^n$ with barycenter at the
origin that satisfies
the regularity conditions (\arabic{eq_935}). Then, with the above notation, for any $x \in \RR^n$,
$$ \Delta \psi(x) \leq 2 R^2(K) $$
where $R(K) = \sup_{x \in K} |x|$ is the outer radius of $K$, and $\Delta \psi$ is the Laplacian of $\psi$.
\label{thm2}
\end{theorem}

Theorem \ref{thm2} is proven by analyzing a certain {\it weighted Riemannian manifold}.
A weighted Riemannian manifold,
sometimes called a {\it Riemannian metric-measure space}, is a triple
$$ X = (\Omega, g, \mu) $$
where $\Omega$ is a smooth manifold (usually an open set in $\RR^n$), where $g$ is a Riemannian metric on $\Omega$,
and $\mu$ is a measure on $\Omega$ with a smooth density with respect to the Riemannian volume measure.
In this paper we study the weighted Riemannian manifold
\begin{equation}  M^*_\mu = \left(\RR^n, \nabla^2 \psi , e^{-\psi(x)} dx \right). \label{eq_959} \end{equation}
That is, the measure associated with $M^*_{\mu}$ has density $e^{-\psi}$ with respect to the Lebesgue measure on $\RR^n$,
and the Riemannian tensor on $\RR^n$ which is induced by the Hessian of $\psi$ is
\begin{equation} \sum_{i,j=1}^n \psi_{ij} dx^i dx^j, \label{eq_957} \end{equation}
where we abbreviate $\psi_{ij} = \partial^2 \psi / \partial x^i \partial x^j$.
There is also a dual description of $M^*_{\mu}$. Recall that the Legendre transform of $f: \RR^n \rightarrow \RR \cup \{ + \infty \}$ is the convex function
$$ f^*(x) = \sup_{y \in \RR^n \atop{f(y) < +\infty}} \left[ x \cdot y - f(y) \right] \quad \quad \quad \quad (x \in \RR^n). $$
We refer the reader to Rockafellar \cite{roc} for the basic properties of the Legendre transform. Denote $\vphi = \psi^*$.
From (\ref{eq_1120}) we see that the Hessian matrix of the convex function $\psi$ is always invertible, hence it is positive-definite.
Therefore $\vphi$ is a smooth function in $K$ whose Hessian is always positive-definite.
Consequently, the map $\nabla \vphi: K \rightarrow \RR^n$ is a diffeomorphism, and $\nabla \psi$ is its inverse map.
One may directly verify
that the weighted Riemannian manifold $M_{\mu}^*$ is canonically isomorphic to
$$ M_{\mu} = \left(K, \nabla^2 \vphi , \mu \right), $$
with $x \mapsto \nabla \psi(x)$ being the isomorphism map. In differential geometry,
the isomorphism between $M_{\mu}$ and $M^*_{\mu}$ is the passage from complex coordinates
to action/angle coordinates, see, e.g., Abreu \cite{abreu}. Here are some basic properties of our weighted Riemannian manifold:
\begin{itemize}

\item[(i)] The space $M_{\mu}$ is stochastically complete. That is, the diffusion process associated with $M_{\mu}$ is well-defined,
it has $\mu$ as a stationary measure and ``it never reaches the boundary of $K$''.
\item[(ii)]
The Bakry-\'Emery-Ricci tensor of $M_{\mu}$ is positive. In fact, it is at least half of  the Riemannian metric tensor.
\item[(iii)] The Laplacian associated with $M_{\mu}$ has an interesting spectrum: The first non-zero eigenvalue is $-1$,
and the corresponding eigenspace contains all linear functions.
\end{itemize}

Property (ii) is a particular case of the results of Kolesnikov \cite[Theorem 4.3]{koles}
(the notation of Kolesnikov is related to ours via $V = \Phi = \psi$),
and properties (i) and (iii)
are discussed below. It is important to note that the construction of $M_{\mu}$
does not rely on the Euclidean structure, and that in principle we could have replaced
$\RR^n$ with an abstract $n$-dimensional linear space. This is in sharp contrast with the
Riemannian metric-measure space $(\RR^n, | \cdot |, \mu)$ that is frequently used
for the analysis of the log-concave measure $\mu$.

\medskip In the following sections we prove the assertions made in the Introduction, and
as a sample of possible applications, we explain below how to recover
the classical Payne-Weinberger spectral gap inequality \cite{PW}, up to a constant factor:

\begin{corollary} Let $\mu$ be a log-concave probability measure on $\RR^n$ with barycenter at the
origin that satisfies the regularity conditions (\arabic{eq_935}). Then, for any $\mu$-integrable,
smooth function $f: K \rightarrow \RR$,
\begin{equation}
 \int_K f^2 d \mu - \left( \int_K f d \mu \right)^2 \leq 2 R^2(K) \int_{K} |\nabla f|^2 d \mu. \label{eq_931}
\end{equation} \label{cor_920}
\end{corollary}

The constant $2 R^2(K)$ on the right-hand side of (\ref{eq_931}) is not optimal.
In the case where $\mu$ is the uniform probability measure on a convex body $K \subset \RR^n$
with a central symmetry (i.e., $K = -K$),
the best possible constant is $4 R^2(K)/ \pi^2$, see Payne and Weinberger \cite{PW}.

\medskip Throughout this note, a convex body in $\RR^n$ is a bounded, open, convex set. We write $\log$ for the natural logarithm.
A smooth function or a smooth manifold are $C^{\infty}$-smooth. The unit sphere is $S^{n-1} = \{ x \in \RR^n ; |x| = 1 \}$.
The five sections below use a variety of techniques, from It\^o calculus to maximum principles. We tried
to make each section as independent of the others as possible.

\medskip
{\it Acknowledgements.} The author would like to thank Bo Berndtsson, Dario Cordero-Erausquin,
Ronen Eldan, Alexander Kolesnikov, Eveline Legendre, Emanuel Milman, Ron Peled, Yanir Rubinstein and Boris Tsirelson
for interesting discussions related to this work.
Supported
by a grant from the European Research Council.

\section{Continuity of the moment measure}
\label{sec_continuity}

This section is concerned with the continuity of the correspondence between convex functions
and their moment measures. Our main result here is Proposition \ref{prop_1130} below.
We say that a convex function $\psi: \RR^n \rightarrow \RR$ is {\it centered} if
\begin{equation}
 \int_{\RR^n} e^{-\psi(x)} dx = 1, \qquad \qquad \int_{\RR^n} x_i e^{-\psi(x)} dx = 0, \ i=1,\ldots,n. \label{eq_1037} \end{equation}
The role of the barycenter condition in (\ref{eq_1037})
is  to prevent translations of $\psi$ which result in the same moment measure.
It is well-known that any convex function $\psi: \RR^n \rightarrow \RR$ satisfying $\int e^{-\psi} = 1$
must tend to $+\infty$
at infinity. More precisely, for any such convex function $\psi$ there exist $A, B > 0$ with
\begin{equation}
 \psi(x) \geq A|x| - B \qquad \qquad (x \in \RR^n), \label{eq_904_}
 \end{equation}
see, e.g., \cite[Lemma 2.1]{K_psi}).

\begin{proposition}
Let $\Omega \subset \RR^n$ be a compact set, and let $\psi, \psi_1,\psi_2,\ldots: \RR^n \rightarrow \RR$
be centered, convex functions.
Denote by  $\mu, \mu_1,\mu_2,\ldots$ the corresponding moment measures, which are
assumed to  be supported in $\Omega$. Then the following are equivalent:
\begin{enumerate}
\item[(i)] $\displaystyle \psi_{\ell} \longrightarrow \psi$ pointwise in $\RR^n$.
\item[(ii)] $\displaystyle \mu_{\ell} \longrightarrow \mu$ weakly (i.e.,
$ \int b d \mu_{\ell} \rightarrow \int b d \mu $
for any continuous function $b: \Omega \rightarrow \RR$).
\end{enumerate} \label{prop_1130}
\end{proposition}

Several lemmas are required for the proof of Proposition \ref{prop_1130}.
For a centered, convex function $\psi: \RR^n \rightarrow \RR$ we define
$$ K(\psi) = \left \{ x \in \RR^n \, ; \, \psi(x) \leq 2n + \inf_{y \in \RR^n} \psi(y) \right \}, $$
a convex set in $\RR^n$. Since the barycenter of $e^{-\psi(x)} dx$ lies at the origin, then $\psi(0) \leq n + \inf_{x \in \RR^n} \psi(x)$,
according to Fradelizi \cite{fradelizi}. Hence the origin is necessarily in the interior of $K(\psi)$. For $x \in \RR^n$ consider the Minkowski functional
$$ \| x \|_{\psi} = \inf \left \{ \lambda > 0 ; x / \lambda \in K(\psi) \right \}. $$
Since a convex function is continuous, then $\psi(x / \| x \|_{\psi}) = 2n + \inf \psi$ for any $0 \neq x \in \RR^n$.
The following lemma is well-known, but nevertheless its proof is provided for completeness.

\begin{lemma} Let $\psi: \RR^n \rightarrow \RR$ be a centered, convex function. Then,
\begin{equation}
\psi(x) \geq n \| x \|_{\psi} + \psi(0) - 2n \quad \quad \quad \quad (x \in \RR^n). \label{eq_2059}
\end{equation}
 \label{lem_1408}
\end{lemma}

\begin{proof} Since the barycenter of $e^{-\psi(x)} dx$ lies at the origin,
\begin{equation}
\psi(0) \leq n + \inf_{x \in \RR^n} \psi(x). \label{eq_1012} \end{equation}
Whenever $x \in K(\psi)$ we have $\| x \|_{\psi} \leq 1$.
 Therefore (\ref{eq_2059}) follows from (\ref{eq_1012}) for $x \in K(\psi)$.
In order to prove (\ref{eq_2059}) for $x \not \in K(\psi)$, we observe that for such $x$ we have $\| x \|_{\psi} \geq 1$ and hence
$$ \psi(0) + n \leq \inf_{y \in \RR^n} \psi(y) + 2n = \psi \left(\frac{x}{\| x \|_{\psi}} \right)
\leq \left( 1- \frac{1}{\| x \|_{\psi}} \right) \cdot \psi(0) + \frac{1}{\| x\|_{\psi}} \cdot \psi(x), $$
due to the convexity of $\psi$. We conclude that $\psi(x) \geq \psi(0) + n \| x \|_{\psi}$ for any $x \not \in K(\psi)$,
and (\ref{eq_2059}) is proven in all cases.
\end{proof}

\begin{proof}[Proof of the direction $\text{(i)} \Rightarrow \text{(ii)}$ in Proposition \ref{prop_1130}]
Denote $$ K = \{ x\in \RR^n ; \psi(x) < 2n+1 + \psi(0) \}, $$ an open, convex set containing the origin. Since $e^{-\psi}$
is integrable, then $K$ must be of finite volume, hence bounded. According to Rockafellar \cite[Theorem 10.8]{roc}, the convergence of $\psi_{\ell}$
to $\psi$ is locally uniform in $\RR^n$. In particular, the convergence is uniform on $K$. Setting $M = \psi(0) - 1$
we conclude that there exists $\ell_0 \geq 1$ such that
\begin{equation}
K(\psi_{\ell}) \subseteq K, \quad \psi_{\ell}(0) \geq M \quad \quad \quad \quad \text{for all} \ \ell \geq \ell_0. \label{eq_901}
\end{equation}
Denote $R = \sup_{x \in K} |x|$. From (\ref{eq_901}) and Lemma \ref{lem_1408}, for any $\ell \geq \ell_0$,
\begin{equation}
\psi_{\ell}(x) \geq n \| x \|_{\psi_{\ell}} + \psi_{\ell}(0) - 2n \geq \frac{n}{R} |x| + (M -2n) \quad  \quad \qquad (x \in \RR^n).
\label{eq_1504_} \end{equation}
According to  our assumption (i) and \cite[Theorem 24.5]{roc} we have that
$$ \nabla \psi_{\ell}(x) \stackrel{\ell \rightarrow \infty}\longrightarrow \nabla \psi(x) $$
for any $x \in \RR^n$ in which $\psi, \psi_1,\psi_2,\ldots$ are differentiable.
Let $b: \Omega \rightarrow \RR$ be a continuous function.
Since a convex function is differentiable almost everywhere, we conclude that
$$  b(\nabla \psi_{\ell}(x)) e^{-\psi_\ell(x)} \stackrel{\ell \rightarrow \infty} \longrightarrow
b(\nabla \psi(x)) e^{-\psi(x)}   \quad \text{for almost any} \ x \in \RR^n.
$$
The function $b$ is bounded because $\Omega$ is compact.
We may use the dominated convergence theorem, thanks to (\ref{eq_1504_}), and conclude that
$$ \int_{\Omega} b d \mu_{\ell} = \int_{\RR^n}  b(\nabla \psi_{\ell}(x)) e^{-\psi_\ell(x)} dx \stackrel{\ell \rightarrow \infty} \longrightarrow
\int_{\RR^n} b(\nabla \psi(x)) e^{-\psi(x)} dx = \int_{\Omega} b d \mu.
$$
Thus (ii) is proven.
\end{proof}

It still remains to prove the direction $\text{(ii)} \Rightarrow \text{(i)}$ in Proposition \ref{prop_1130}.
A function $f: \RR^n \rightarrow \RR$ is $L$-Lipschitz if $|f(x) - f(y)| \leq L|x-y|$ for any $x,y \in \RR^n$.

\begin{lemma} Let $L, \eps > 0$.
Suppose that $\psi: \RR^n \rightarrow \RR$ is a centered, $L$-Lipschitz, convex function, such that
\begin{equation}
 \int_{\RR^n} |\nabla \psi(x) \cdot \theta| e^{-\psi(x)} dx \geq \eps \quad \quad \quad \quad \text{for all} \ \theta \in S^{n-1}.
 \label{eq_1026} \end{equation}
Then,
\begin{equation}
\alpha |x| - \beta \leq \psi(x) \leq L |x| + \gamma \quad \quad \quad \quad (x \in \RR^n),\label{eq_1418}
  \end{equation}
where $\alpha, \beta, \gamma > 0$ are constants depending only
on $L, \eps$ and $n$.
 \label{lem_1803}
\end{lemma}

\begin{proof} Fix $\theta \in S^{n-1}$ and
 set $H = \theta^{\perp}$, the hyperplane orthogonal to $\theta$. The function
 $$ m_{\theta}(y) = \inf_{t \in \RR} \psi(y + t \theta) \quad \quad \quad \quad (y \in H) $$
 is convex.
 Furthermore, for any fixed $y \in H$, the function $t \mapsto \psi(y + t \theta)$
 is convex, $L$-Lipschitz and tends to $+\infty$ as $t \rightarrow \pm \infty$. Hence the one-dimensional
 convex function $t \mapsto \psi(y + t \theta)$ attains its minimum at a certain point $t_0 \in \RR$, is non-decreasing on $[t_0, +\infty)$ and non-increasing
 on $(-\infty, t_0]$. Therefore, for any $y \in H$,
 $$ \int_{-\infty}^{\infty} \left| \frac{\partial \psi(y + t \theta)}{\partial t}  \right| e^{-\psi(y + t \theta)} dt = \int_{-\infty}^{\infty} \left| \frac{\partial}{\partial t} e^{-\psi(y + t \theta)} \right| dt  = 2 e^{-m_{\theta}(y)}. $$
We now integrate over $y \in H$ and use Fubini's theorem to conclude that
\begin{equation}
\int_{\RR^n} |\nabla \psi(x) \cdot \theta| e^{-\psi(x)} dx = 2 \int_H e^{-m_{\theta}(y)} dy. \label{eq_1134_}
\end{equation}
 Consider the interval
\begin{equation}
 I_{\theta} = \left \{ t \in \RR \, ; \, t \theta \in K(\psi) \right \}. \label{eq_1410}
 \end{equation}
Then,
\begin{equation}
\int_{-\infty}^{\infty} e^{-\psi \left( t\theta \right) /2} dt \geq
\int_{I_{\theta}} e^{-\psi \left( t\theta \right) /2} dt \geq
e^{-n-\frac{m_{\theta}(0)}{2}}  |I_{\theta}|
\label{eq_1157_}
\end{equation}
 where $|I_{\theta}|$ is the length of the interval $I_{\theta}$.
 Fix a point $y \in H$. Then there exists $t_0 \in \RR$ for which $m_{\theta}(y) = \psi(y + t_0 \theta)$.
From (\ref{eq_1157_}) and from the convexity of $\psi$,
\begin{align} \nonumber
\int_{-\infty}^{\infty}  e^{ -\psi \left(\frac{y}{2} + t \theta
\right) } dt & = \frac12 \int_{-\infty}^{\infty} e^{ -\psi \left( \frac{y + t_0 \theta}{2}+ \frac{t\theta}{2} \right) } dt
 \, \geq \, \frac12 e^{-\frac{m_{\theta}(y)}{2}}
\int_{-\infty}^{\infty} e^{-\frac{\psi \left( t\theta \right)}{ 2}} dt \\ &
\geq \frac12 e^{-\frac{m_{\theta}(y) + m_{\theta}(0)}{2}} e^{-n} |I_{\theta}|
\, \geq \, \frac12 e^{-m_{\theta}(y)} e^{-2n} |I_{\theta}|, \label{eq_938}
\end{align}
 where in the last passage we used the fact that $m_{\theta}(0) \leq \psi(0) \leq n + \inf \psi \leq n + m_{\theta}(y)$,
  because the barycenter of $e^{-\psi(x)} dx$ lies at the origin.
 Integrating (\ref{eq_938}) over $y \in H$, we see that
 $$ \int_{H} e^{-m_{\theta}(y)} dy \leq \frac{2e^{2n}}{|I_\theta|} \int_{H}
 \int_{-\infty}^{\infty} e^{-\psi \left(\frac{y}{2} + t \theta
 \right) }
dt dy = \frac{2^n e^{2n}}{|I_\theta|} \int_{\RR^n} e^{-\psi} = \frac{2^n e^{2n}}{|I_\theta|}. $$
Combine the last inequality with (\ref{eq_1026}) and (\ref{eq_1134_}). This leads to the bound
\begin{equation}
 |I_\theta| \leq C_n \left( \int_{\RR^n} |\nabla \psi(x) \cdot \theta| e^{-\psi(x)} dx \right)^{-1} \leq \frac{C_n}{\eps}, \label{eq_1411_}
 \end{equation}
 for some constant $C_n$ depending only on $n$. Recall that the origin belongs to $K(\psi)$
 and hence $0 \in I_{\theta}$.
 By letting $\theta$ range over all of $S^{n-1}$ and glancing at (\ref{eq_1410}) and (\ref{eq_1411_}), we see that
 \begin{equation}
  K(\psi) \subseteq  B \left(0, C_n / {\eps} \right) \label{eq_2112_}
  \end{equation}
 where $B(x,r) = \{ y \in \RR^n ; |y-x| \leq r \}$. From (\ref{eq_2112_}) and from Lemma \ref{lem_1408},
 \begin{equation}
 \psi(x) \geq \psi(0) - 2n + n \| x \|_{\psi} \geq \psi(0) - 2n + \frac{\eps}{\tilde{C}_n} |x|
 \qquad \qquad (x \in \RR^n),
 \label{eq_2112} \end{equation}
 for $\tilde{C}_n = C_n / n$. By integrating (\ref{eq_2112}) we obtain
 $$ 1 = \int_{\RR^n} e^{-\psi} \leq e^{-(\psi(0) - 2n)} \int_{\RR^n} e^{-\eps |x| / \tilde{C}_n} dx . $$
Therefore, $\psi(0) \leq \gamma$ for $\gamma = 2n + \log(\int_{\RR^n} e^{-\eps |x| / \tilde{C}_n} dx)$.
Since $\psi$ is $L$-Lipschitz, then the right-hand side inequality of (\ref{eq_1418}) follows. Next, observe that
$$
 1 = \int_{\RR^n} e^{-\psi(x)} dx \geq \int_{\RR^n} e^{-\psi(0) - L |x|} dx = e^{-\psi(0)} \int_{\RR^n} e^{ - L |x|} dx.
 $$
Hence $\psi(0) \geq \log(\int_{\RR^n} e^{-L |x|} dx)$,
and the left-hand side inequality of (\ref{eq_1418}) follows from (\ref{eq_2112}).
\end{proof}

\begin{proof}[Proof of the direction $\text{(ii)} \Rightarrow \text{(i)}$ in Proposition \ref{prop_1130}] $ $

\medskip {\bf Step 1.} We claim that
\begin{equation}
\liminf_{\ell \rightarrow \infty} \left( \inf_{\theta \in S^{n-1}} \int_{\Omega} |x \cdot \theta| d \mu_{\ell}(x) \right) > 0.
\label{eq_1724} \end{equation}
Assume that (\ref{eq_1724}) fails. Then there exist sequences $\ell_j \in \NN$ and $\theta_j \in S^{n-1}$ such that
\begin{equation}
 \lim_{j \rightarrow \infty} \int_{\Omega} |x \cdot \theta_{j}| d \mu_{\ell_j}(x) = 0.\label{eq_2159}
 \end{equation}
Passing to a subsequence, if necessary, we may assume that $\theta_j \longrightarrow \theta_0 \in S^{n-1}$.
The sequence
of functions $|x \cdot \theta_{j}|$ tends to $|x \cdot \theta_0|$ uniformly in $x \in \Omega$. Hence, from (ii) and (\ref{eq_2159}),
$$
 \int_{\Omega} |x \cdot \theta_0| d \mu(x) = \lim_{j \rightarrow \infty} \int_{\Omega} |x \cdot \theta_{0}| d \mu_{\ell_j}(x) = \lim_{j \rightarrow \infty} \int_{\Omega} |x \cdot \theta_{j}| d \mu_{\ell_j}(x) = 0.
 $$
Therefore $\mu$ is supported in the hyperplane $\theta_0^{\perp}$. However, $\mu$ is the moment measure of the convex
function $\psi: \RR^n \rightarrow \RR$,
and according to \cite[Proposition 1]{CK}, it cannot be supported in a hyperplane. We have thus arrived at a contradiction, and (\ref{eq_1724}) is proven.

\medskip  {\bf Step 2.} We will prove that there exist $\alpha, \beta, \gamma > 0$ and $\ell_0 \geq 1$ such that
\begin{equation}
\alpha |x| - \beta \leq \psi_{\ell}(x) \leq L |x| + \gamma \quad \quad \quad \quad (\ell \geq \ell_0, x \in \RR^n). \label{eq_1803}
  \end{equation}
Indeed, according to Step 1, there exists $\ell_0 \geq 1$ and $\eps_0 > 0$ such that
\begin{equation}
\int_{\RR^n} |\nabla \psi_{\ell}(x) \cdot \theta| e^{-\psi_{\ell}(x)} dx = \int_{\Omega} |x \cdot \theta| d \mu_{\ell}(x) > \eps_0 \quad \quad (\ell \geq \ell_0, \theta \in S^{n-1}). \label{eq_1801}
\end{equation}
Denote $L = \sup_{x \in \Omega} |x|$. The function $\psi_{\ell}$ is centered and convex. Furthermore, for almost any $x \in \RR^n$
we know that $\nabla \psi_{\ell}(x) \in \Omega$, because the moment measure of $\psi_{\ell}$ is supported in $\Omega$. Hence, for $\ell \geq 1$,
\begin{equation}
 |\nabla \psi_{\ell}(x)| \leq L \qquad \qquad \text{for almost any} \ x \in \RR^n. \label{eq_1031} \end{equation}
Since a convex function is always locally-Lipschitz, then (\ref{eq_1031}) implies that $\psi_{\ell}$ is $L$-Lipschitz,
for any $\ell$. We may now apply Lemma \ref{lem_1803}, thanks to (\ref{eq_1801}), and conclude
(\ref{eq_1803}).

\medskip {\bf Step 3.} Assume by contradiction that there exists $x_0 \in \RR^n$ for
which $\psi_{\ell}(x_0)$ does not converge to $\psi(x_0)$. Then there exist $\eps > 0$ and a subsequence $\ell_j$ such that
\begin{equation}
|\psi_{\ell_j}(x_0) - \psi(x_0)| \geq \eps \quad \quad \quad \quad (j=1,2,\ldots).
\label{eq_1819}
\end{equation}
From (\ref{eq_1803}) we know that the sequence of functions $\{ \psi_{\ell_j} \}_{j=1,2,\ldots}$ 
is uniformly bounded on any compact subset of $\RR^n$.  Furthermore, $\psi_{\ell_j}$ is $L$-Lipschitz for any $j$.
According to the Arzel\'a-Ascoli theorem, we may pass to a subsequence and assume that $\psi_{\ell_j}$ converges locally uniformly in $\RR^n$,
to a certain function $F$. The function $F$ is convex and $L$-Lipschitz, as it is the limit of convex and $L$-Lipschitz functions.
Furthermore, thanks to (\ref{eq_1803}) we may apply the dominated convergence theorem and conclude that $F$ is centered.

\medskip To summarize, the functions $F, \psi_{\ell_1}, \psi_{\ell_2},\ldots$ are $L$-Lipschitz, centered and convex.
We know that $\psi_{\ell_j} \longrightarrow F$ locally uniformly in $\RR^n$. According to the implication
$\text{(i)} \Rightarrow \text{(ii)}$ proven above, we know that $\mu_{\ell_j}$ converges weakly to the moment
measure of $F$. But we assumed that $\mu_{\ell_j}$ converges weakly to $\mu$, and hence $\mu$ is the moment
measure of $F$. Thus $\psi,F: \RR^n \rightarrow \RR$ are two centered, convex functions with the same moment
measure $\mu$. This means that $\psi \equiv F$, according to the uniqueness part in \cite{CK}. Therefore $\psi_{\ell_j} \longrightarrow
\psi$ pointwise in $\RR^n$, in contradiction to (\ref{eq_1819}), and the proof is complete.
\end{proof}

\section{A preliminary weak bound using the maximum principle}
\label{sec_apriori}

In this section we prove a rather weak form of Theorem \ref{thm2}, which will be needed for
the proof of the theorem later on in Section \ref{sec_BE}. Throughout
this section, $\mu$ is a log-concave probability measure on $\RR^n$ with barycenter at the origin, supported on a convex
body $K \subset \RR^n$, with density $e^{-\rho}$ satisfying the regularity conditions (\arabic{eq_935}).
Also, $\psi: \RR^n \rightarrow \RR$ is the smooth, convex function whose moment measure is $\mu$, which is uniquely defined
up to translation, and $\vphi = \psi^*$ is its Legendre transform.
In this section we make the following strict-convexity assumptions:
\begin{enumerate}
\item[($\star$)] The convex body $K$ has a smooth boundary
and its Gauss curvature is positive everywhere. Additionally, there exists $\eps_0 > 0$ with
\begin{equation}  \nabla^2 \rho(x)  \geq \eps_0 \cdot Id \quad \quad \quad \quad (x \in K), \label{eq_551} \end{equation}
in the sense of symmetric matrices.
\end{enumerate}
Denote by $\| A \|$ the operator norm of the matrix $A$.
Our goal in this section is to prove the following:

\begin{proposition} Under the above assumptions,
$$  \sup_{x \in \RR^n} \|  \nabla^2 \psi(x) \|  < +\infty. $$
\label{prop_1101}
\end{proposition}

The argument we present for the demonstration of Proposition \ref{prop_1101} closely follows  the proof of Caffarelli's contraction theorem \cite[Theorem 11]{caf_FKG}.
An alternative approach to Proposition \ref{prop_1101} is outlined in Kolesnikov \cite[Section 6]{koles2}.
We begin the proof of Proposition \ref{prop_1101} with the following lemma, which is due to Berman and Berndtsson \cite{BB}. Their proof is reproduced here for completeness.

\begin{lemma} $\displaystyle \sup_{x \in K} \vphi(x) < +\infty$.
\label{lem_1146}
\end{lemma}

\begin{proof} Since $K$ is bounded, it suffices to show that $\vphi$ is $\alpha$-H\"older for some $\alpha > 0$. According
to the Sobolev inequality in the convex domain $K \subset \RR^n$ (see, e.g., \cite[Chapter 1]{SC}), it is sufficient to prove that
\begin{equation}
\int_K |\nabla \vphi(x)|^p dx < +\infty, \label{eq_1130}
\end{equation}
for some $p > n$. Fix $p > n$. The map $x \mapsto \nabla \vphi(x)$ pushes the measure $\mu$ forward to $\exp(-\psi(x)) dx$. Hence,
\begin{equation}
 \int_K |\nabla \vphi|^p d \mu = \int_{\RR^n} |x|^p e^{-\psi(x)} dx < + \infty, \label{eq_1134}
 \end{equation}
where we used the fact that $e^{-\psi}$ decays exponentially at infinity (see, e.g., (\ref{eq_904_}) above or \cite[Lemma 2.1]{K_psi}).
Since $\rho$ is a bounded function on $K$ and $e^{-\rho}$ is the density of $\mu$, then (\ref{eq_1130}) follows from (\ref{eq_1134}).
\end{proof}

For $x \in \RR^n$ denote $h_K(x) = \sup_{y \in K} x \cdot y$, the supporting functional of $K$. The following lemma is analogous to
\cite[Lemma 4]{caf_FKG}.

\begin{lemma}
$ \displaystyle \lim_{R \rightarrow \infty} \sup_{|x| \geq R} \left| \nabla \psi(x) - \nabla h_K(x) \right| = 0.$
\label{lem_1207}
\end{lemma}

\begin{proof} The function $\vphi: K \rightarrow \RR$ is convex, hence bounded from below by some affine function, which
in turn is greater than some constant on the bounded set $K$. According to Lemma \ref{lem_1146}, the function $\vphi$ is also bounded from above.
Set $M = \sup_{x \in K} |\vphi(x)|$. By elementary properties of the Legendre transform, for any $x \in \RR^n$,
\begin{equation}  \psi(x) =  x \cdot \nabla \psi(x) - \vphi(\nabla \psi(x)) \leq x \cdot \nabla \psi(x) + M.
\label{eq_1151} \end{equation}
However, for any $x \in \RR^n$,
\begin{equation} \psi(x) = \sup_{y \in K} \left[ x \cdot y - \vphi(y) \right] \geq -M + \sup_{y \in K} x \cdot y = -M + x \cdot \nabla h_K(x), \label{eq_1152}
\end{equation}
as $\nabla h_K(x) \in \partial K$ is the unique point at which $\sup_{y \in K} x \cdot y$ is attained.
Using (\ref{eq_1151}) and (\ref{eq_1152}),
\begin{equation}
 (\nabla h_K(x) - \nabla \psi(x)) \cdot \frac{x}{|x|} \leq \frac{2M}{|x|} \quad \quad \quad \quad (0 \neq x \in \RR^n). \label{eq_1155_}
 \end{equation}
 Recall that $\nabla \psi(x) \in K$ for any $x \in \RR^n$.
 Since $\partial K$  is smooth with positive Gauss curvature, inequality (\ref{eq_1155_}) implies that there
exist $R_K, \alpha_K > 0$, depending only on $K$,  with
\begin{equation}
 |\nabla h_K(x) - \nabla \psi(x)| \leq \alpha_K \sqrt{ \frac{ 2M}{|x|} } \quad \quad \quad \quad \text{for} \ |x| \geq R_K.
 \label{eq_1157}
 \end{equation}
The lemma follows from (\ref{eq_1157}).
\end{proof}

For $\eps > 0, \theta \in \RR^n$ and a function $f: \RR^n \rightarrow \RR$ denote
$$ \delta^{(\eps)}_{\theta \theta} f (x) =  f(x + \eps \theta) + f(x - \eps \theta) - 2 f(x) \quad \quad \quad \quad (x \in \RR^n). $$
For a smooth $f$ and a small $\eps$, the quantity $\delta^{(\eps)}_{\theta \theta} f (x) / \eps^2$ approximates
the pure second derivative $f_{\theta \theta}(x)$.
We would like to use the maximum principle for the function $\psi_{\theta \theta}(x)$, but
we do not know whether or not it attains its supremum. This is the reason for using the approximate second derivative $\delta^{(\eps)}_{\theta \theta} \psi (x)$
as a substitute.

\begin{corollary} Fix $0 < \eps < 1$. Then the supremum of $\delta^{(\eps)}_{\theta \theta} \psi(x) $
over all $x \in \RR^n$ and $\theta \in S^{n-1}$ is attained. \label{cor_1304}
\end{corollary}

\begin{proof} According to Lemma \ref{lem_1207} and the continuity and $0$-homogeneity of $\nabla h_K(x)$,
\begin{align} \nonumber
& \lim_{R \rightarrow \infty}  \sup_{|x| \geq R \atop{x_1,x_2 \in B(x, 1)}}  |\nabla \psi(x_1) - \nabla \psi(x_2)|  =
\lim_{R \rightarrow \infty}  \sup_{|x| \geq R \atop{x_1,x_2 \in B(x, 1)}} |\nabla h_K(x_1) - \nabla h_K(x_2)| \nonumber
\\
& =  \lim_{R \rightarrow \infty} \sup_{|x| = 1 \atop{x_1,x_2 \in B(x, 1/R)}} |\nabla h_K(x_1) - \nabla h_K(x_2)| = 0,
\label{eq_1109}
\end{align}
where $B(x, r) = \left \{ y \in \RR^n ; |x - y| < r \right \}$.
From Lagrange's mean value theorem,
\begin{align} \nonumber \delta^{(\eps)}_{\theta \theta} \psi (x) & =  ( \psi(x + \eps \theta) - \psi(x) ) \, - \, ( \psi(x) - \psi(x - \eps \theta)  )
\\ & \leq \eps \sup_{x_1,x_2 \in B(x,\eps)} |\nabla \psi(x_1) - \nabla \psi(x_2)|. \label{eq_1106} \end{align}
According to (\ref{eq_1109}) and (\ref{eq_1106}),
\begin{align}
\lim_{R \rightarrow \infty} & \sup_{|x| \geq R \atop{\theta \in S^{n-1}}} \delta^{(\eps)}_{\theta \theta} \psi (x)
\leq \eps \lim_{R \rightarrow \infty} \sup_{|x| \geq R \atop{x_1,x_2 \in B(x, \eps)}} |\nabla \psi(x_1) - \nabla \psi(x_2)| = 0.\label{eq_1115}
\end{align}
Since $\psi$ is convex and smooth, then the function $\delta^{(\eps)}_{\theta \theta} \psi$ is non-negative and continuous in $(x, \theta) \in \RR^n \times S^{n-1}$.
It thus follows from (\ref{eq_1115}) that its supremum is attained.
\end{proof}

We shall apply the well-known matrix inequality, which states that when $A$ and $B$ are symmetric, positive-definite $n \times n$ matrices,
then
\begin{equation}
 \log \det B \leq \log \det A + Tr \left[ A^{-1}(B-A) \right] = \log \det A + Tr \left[ A^{-1}B \right] - n, \label{eq_1315}
 \end{equation}
 where $Tr(A)$ stands for the trace of the matrix $A$.
Recall that the transport equation (\ref{eq_1120}) is valid, hence,
\begin{equation}
\log \det \nabla^2 \psi(x) = -\psi(x) + (\rho \circ \nabla \psi)(x) \quad \quad \quad \quad (x \in \RR^n).
\label{eq_1155}
\end{equation}
In particular, $\nabla^2 \psi(x)$ is always an invertible
matrix  which is in fact positive-definite. We denote its inverse by $\left( \nabla^2 \psi(x) \right)^{-1} = ( \psi^{ij}(x) )_{i,j=1,\ldots,n}$.
For a smooth function
$u: \RR^n \rightarrow \RR$ denote
\begin{equation}
A u(x)  = Tr \left[ \left(\nabla^2 \psi(x) \right)^{-1} \nabla^2 u(x) \right] = \psi^{ij}(x) u_{ij}(x) \quad \quad (x \in \RR^n),
\label{eq_1311}
\end{equation}
where we adhere to the Einstein convention: When an index is repeated twice
in an expression, once as a subscript and once as a superscript, then we sum over this index from $1$ to $n$.
According to (\ref{eq_1315}) for any $\theta \in \RR^n$,
\begin{equation}
 \log \det \nabla^2 \psi(x + \theta) \leq \log \det \nabla^2 \psi(x) + \psi^{ij}(x) \psi_{ij}(x + \theta) - n
 \qquad (x \in \RR^n), \label{eq_1317}
 \end{equation}
with an equality for $\theta =0$.

\begin{proof}[Proof of Proposition  \ref{prop_1101}] We follow Caffarelli's argument \cite[Theorem 11]{caf_FKG}. Our assumption (\ref{eq_551})
yields that the function $\rho(x) - \eps_0 |x|^2/2$ is convex. Hence, for any $x,y$ such that $x-y,x+y,x \in K$,
\begin{equation}
\rho(x + y) + \rho(x - y) - 2 \rho(x) \geq \frac{\eps_0}{2} \left( |x+y|^2 + |x-y|^2 - 2|x|^2 \right) = \eps_0 |y|^2.
\label{eq_554}
\end{equation}
Fix $0 < \eps <1$ and abbreviate $\delta_{\theta \theta} f = \delta^{(\eps)}_{\theta \theta} f$.
From (\ref{eq_1155}) and (\ref{eq_1317}) as well as some simple algebraic manipulations,
for any $\theta \in \RR^n$,
\begin{equation}
 A( \delta_{\theta \theta} \psi) \geq \delta_{\theta \theta} \left( \log \det \nabla^2 \psi \right) = -\delta_{\theta \theta} \psi + \delta_{\theta \theta} (\rho \circ \nabla \psi). \label{eq_534}
 \end{equation}
According to Corollary \ref{cor_1304}, the maximum of $(x, \theta) \mapsto \delta_{\theta \theta}\psi(x)$ over $\RR^n \times S^{n-1}$
is attained at some $(x_0, e) \in \RR^n \times S^{n-1}$.
Since $\psi$ is smooth, then at the point $x_0$,
$$  0 = \nabla (\delta_{ee} \psi)(x_0) = \nabla \psi(x_0 + \eps e) + \nabla \psi(x_0 + \eps e) - 2 \nabla \psi(x_0).
$$
In other words, there exists a vector $u \in \RR^n$ such that
$$ \nabla \psi(x_0 + \eps e) = \nabla \psi(x_0) + u, \quad \quad \nabla \psi(x_0 - \eps e) = \nabla \psi(x_0) - u. $$
 Setting $v = \nabla \psi(x_0)$ and using (\ref{eq_554}), we obtain
\begin{equation}  \delta_{ee} (\rho \circ \nabla \psi)(x_0) = \rho(v + u) + \rho(v- u) -2 \rho(v) \geq \eps_0 |u|^2.
\label{eq_555} \end{equation}
The smooth function $x \mapsto \delta_{ee} \psi(x)$ reaches a maximum at $x_0$,
hence the matrix $\nabla^2 \left( \delta_{ee} \psi \right)(x_0)$ is negative semi-definite. Since the matrix $(\nabla^2 \psi)^{-1}(x_0)$ is positive-definite,
then from the definition (\ref{eq_1311}),
\begin{equation}  0 \geq A (\delta_{ee} \psi) (x_0).
\label{eq_556} \end{equation}
Now, (\ref{eq_534}), (\ref{eq_555}) and (\ref{eq_556}) yield
\begin{equation}  \delta_{ee} \psi(x_0) \geq \delta_{ee} \left( \rho \circ \nabla \psi \right)(x_0) \geq \eps_0 |u|^2. \label{eq_559}
\end{equation}
By the convexity of $\psi$,
$$ \psi(x_0 + \eps e) - \psi(x_0) \leq \nabla \psi(x_0 + \eps e) \cdot (\eps e) = (v + u) \cdot (\eps e) $$
and
$$ \psi(x_0 - \eps e) - \psi(x_0) \leq \nabla \psi(x_0 - \eps e) \cdot (-\eps e) = (v - u) \cdot (-\eps e). $$
Summing the last two inequalities yields
\begin{equation}  \delta_{ee}\psi(x_0) \leq (v + u) \cdot (\eps e) + (v - u) \cdot (-\eps e) = 2 \eps (u \cdot e) \leq 2 |u| \eps. \label{eq_602}
\end{equation}
The inequalities (\ref{eq_559}) and (\ref{eq_602}) imply that $|u| \leq 2 \eps / \eps_0$ and hence from (\ref{eq_602}),
$$ \delta_{ee}(\psi)(x_0) \leq 4 \eps^2 / \eps_0. $$
Consequently, for any $x \in \RR^n$ and $\theta \in S^{n-1}$ we have $\delta_{\theta \theta}^{(\eps)} \psi(x) \leq 4 \eps^2 / \eps_0$, and hence
$$ \psi_{\theta \theta}(x) = \lim_{\eps \rightarrow 0^+} \frac{\delta^{(\eps)}_{\theta \theta} \psi(x)}{\eps^2} \leq
 \frac{4}{\eps_0}. $$
Therefore $\| \nabla^2 \psi(x) \| \leq 4 / \eps_0$ for any $x \in \RR^n$, and the proof is complete.
\end{proof}

\begin{remark} {\rm Our proof of Proposition \ref{prop_1101} provides the explicit bound
\begin{equation}
 \sup_{x \in \RR^n} \| \nabla^2 \psi(x) \| \leq 4 / \eps_0. \label{eq_1203} \end{equation}
By arguing as in \cite{caf_FKG_erat}, one may improve the right-hand side of (\ref{eq_1203}) to just $1/ \eps_0$. We
omit the straightforward details.
}
\end{remark}

\section{Diffusion processes and stochastic completeness}
\label{sec_diffusion}

In this section we consider a diffusion process associated with transportation of measure.
Our point of view
owes much to the article by Kolesnikov \cite{koles}, and we make an effort to maintain a discussion
as general as the one in Kolesnikov's work.

\medskip Let $\mu$ be a probability measure supported
on an open set $K \subseteq \RR^n$, with density $e^{-\rho}$ where $\rho: K \rightarrow \RR$ is a smooth
function. Let $\psi: \RR^n \rightarrow \RR$ be a smooth, convex function with
\begin{equation}
 \lim_{R \rightarrow \infty} \left( \inf_{|x| \geq R} \psi(x) \right) = +\infty. \label{eq_1205}
 \end{equation}
Condition (\ref{eq_1205}) holds automatically when $\int e^{-\psi} <\infty$, see (\ref{eq_904_}) above.
Rather than requiring that
the transport equation (\ref{eq_1120}) hold true, in this section we make
the more general  assumption that
\begin{equation} e^{-\rho(\nabla \psi(x))} \det \nabla^2 \psi(x) = e^{-V(x)}
\quad \quad \quad \quad (x\in \RR^n)
\label{eq_949}
\end{equation}
for a certain smooth function $V: \RR^n \rightarrow \RR$.
 Clearly, when $\mu$ is the moment
measure of $\psi$, equation (\ref{eq_949}) holds true with $V = \psi$ and condition (\ref{eq_1205})
holds as well. The transport
equation (\ref{eq_949}) means that the map $x \mapsto \nabla \psi(x)$ pushes
the probability measure $e^{-V(x)} dx$ forward to $\mu$.
 In this section we explain and prove the following:

\begin{proposition} Let $K \subseteq \RR^n$ be an open set, and let $V, \psi: \RR^n \rightarrow \RR$ and $\rho: K \rightarrow \RR$
be smooth functions with $\psi$ being convex. Assume (\ref{eq_1205}) and (\ref{eq_949}), and furthermore, that
\begin{equation}
\inf_{x \in K} \nabla \rho(x) \cdot x > -\infty. \label{eq_2146}
\end{equation}
Then the weighted Riemannian manifold $M = \left(\RR^n, \nabla^2 \psi, e^{-V(x)} dx \right)$ is stochastically complete.
\label{prop_957}
\end{proposition}

\begin{remark} {\rm
Note that in the most interesting case where $V = \psi$, the weighted Riemannian manifold $M$ from Proposition \ref{prop_957}
coincides with $M_{\mu}^*$ as defined in (\ref{eq_959}) and (\ref{eq_957}) above. Additionally, in the case where
$\mu$ is log-concave with barycenter at the origin, condition (\ref{eq_2146}) does hold true: In this case,
according to Fradelizi \cite{fradelizi}, we know that $\rho(0) \leq n + \inf_{x \in K} \rho(x)$. By convexity,
$$  \nabla \rho(x) \cdot x \geq \rho(x) - \rho(0) \geq -n \quad \quad \quad (x \in K), $$
and (\ref{eq_2146}) follows. Thus Proposition \ref{prop_957} implies the stochastic completeness
of $M_{\mu}^*$ when $\mu$ is a log-concave probability measure with barycenter at the origin,
which satisfies the regularity conditions (\arabic{eq_935}).
}
\end{remark}

\medskip We now turn to a detailed explanation of {\it stochastic completeness} of a weighted Riemannian manifold.
See, e.g., Grigor'yan \cite{gri} for more information.
 The {\it Dirichlet form}  associated with the weighted Riemannian manifold $M = (\Omega, g, \nu)$ is defined as
\begin{equation}  \Gamma(u, v) = \int_{\Omega} g \left( \nabla_g u, \nabla_g v \right) d \nu, \label{eq_1015} \end{equation}
where $u,v: \Omega \rightarrow \RR$ are smooth functions for which the integral in (\ref{eq_1015}) exists.
Here, $\nabla_g u$ stands for the Riemannian gradient of $u$. The {\it Laplacian} associated with $M$
is the unique operator $L$, acting on smooth functions $u: \Omega \rightarrow \RR$, for which
\begin{equation}  \int_{\Omega} (Lu) v d \nu = -\Gamma(u,v) \label{eq_226} \end{equation}
for any compactly-supported, smooth function $v: \Omega \rightarrow \RR$. In the case of the weighted manifold  $M = \left(\RR^n, \nabla^2 \psi, e^{-V(x)} dx \right)$ from Proposition \ref{prop_957},
we may express the Dirichlet form as follows:
\begin{equation}
 \Gamma(u,v) = \int_{\RR^n} \left( \psi^{ij} u_i v_j \right) e^{-V} \label{eq_947} \end{equation}
where $\nabla^2 \psi(x)^{-1} = ( \psi^{ij}(x) )_{i,j=1,\ldots,n}$
and $u_i = \partial u / \partial x^i$. Note that the matrix $\nabla^2 \psi(x)$ is invertible, thanks to (\ref{eq_949}).
As in Section \ref{sec_apriori} above, we use the Einstein summation convention; thus
in (\ref{eq_947}) we sum over $i,j$ from $1$ to $n$.
We will also make use of abbreviations such as $\psi_{ijk} = \partial^3 \psi / (\partial x^i \partial x^j \partial x^k)$,
and also $\psi^{i}_{j \ell} = \psi^{i k} \psi_{j k \ell}$ and $\psi^{ij}_k = \psi^{i\ell} \psi^{j m} \psi_{\ell m k}$.
Therefore, for example,
$$ (\psi^{ij})_k = \frac{\partial \psi^{ij}(x)}{\partial x^k} = -\psi^{i\ell} \psi^{j m} \psi_{\ell m k} = -\psi^{ij}_k. $$
We may now express the Laplacian $L$ associated with $M = \left(\RR^n, \nabla^2 \psi, e^{-V(x)} dx \right)$ by
\begin{equation}
 L u = \psi^{ij} u_{ij} - (\psi^{ij}_j + \psi^{ij} V_j) u_i \label{eq_1128}
 \end{equation}
as may be directly verified from (\ref{eq_947}) by integration by parts.

\begin{lemma} For any smooth function $u: \RR^n \rightarrow \RR$,
\begin{equation}
 L u = \psi^{ij} u_{ij} - \sum_{j=1}^n \rho_j(\nabla \psi(x)) u_j.
 \label{eq_950} \end{equation}
\label{lem_950}
\end{lemma}

\begin{proof}
We take  the logarithmic derivative of (\ref{eq_949}) and obtain that for  $\ell=1,\ldots,n$,
\begin{equation}
\psi^{i}_{i \ell}(x) = - V_{\ell}(x) + \sum_{i=1}^n \rho_i(\nabla \psi(x)) \psi_{i \ell}(x) \qquad (x \in \RR^n). \label{eq_1259}
\end{equation}
Multiplying (\ref{eq_1259}) by $\psi^{j \ell}$ and summing over $\ell$
we see that for $j=1,\ldots,n$,
\begin{equation}
\psi^{i j}_{i}(x) = -\psi^{j \ell}(x) V_{\ell}(x) + \rho_j(\nabla \psi(x)) \qquad (x \in \RR^n). \label{eq_1129} \end{equation}
Now (\ref{eq_950}) follows from (\ref{eq_1128}) and (\ref{eq_1129}).
\end{proof}

\begin{lemma} Under the assumptions of Proposition \ref{prop_957}, there exists $A \geq 0$
such that for all $x \in \RR^n$,
$$ (L \psi)(x) \leq A. $$
\label{lem_1038}
\end{lemma}

\begin{proof} Set $A = \max \left \{0, n - \inf_{y \in K} \nabla \rho(y) \cdot y \right \}$, which is a finite number
according to our assumption (\ref{eq_2146}). From Lemma \ref{lem_950},
$$
L \psi(x) = \psi^{ij} \psi_{ij} - \sum_{j=1}^n \rho_j(\nabla \psi(x)) \psi_j(x) =n - \sum_{j=1}^n \rho_j(\nabla \psi(x)) \psi_j(x).
$$
It remains to prove that $n- \sum_j \rho_j(\nabla \psi(x)) \psi_j(x) \leq A$, or equivalently, we need to show that
\begin{equation}
\nabla \rho(y) \cdot y \geq n-A \quad \quad \quad \quad \text{for all} \ y \in K.
\label{eq_1132_}
\end{equation}
However, (\ref{eq_1132_}) holds true in view of the definition of $A$ above. Therefore $L \psi \leq A$ pointwise in $\RR^n$. \end{proof}

The Laplacian $L$ associated with a weighted Riemannian manifold $M$ is a second-order, elliptic
operator with smooth coefficients. We say that $M$ is {\it stochastically complete} if the It\^o diffusion process
whose generator is $L$ is well-defined at all times $t \in [0, \infty)$. In the particular case of Proposition \ref{prop_957},
this means the following: Let $(B_t)_{t \geq 0}$ be the standard $n$-dimensional Brownian motion. The diffusion equation
with generator $L$ as in (\ref{eq_950}) is the stochastic differential equation:
\begin{equation}  d Y_t = \sqrt{2} \left( \nabla^2 \psi(Y_t) \right)^{-1/2} d B_t - \nabla \rho(\nabla \psi(Y_t)) dt,
\label{eq_342} \end{equation}
where $(\nabla^2 \psi(x))^{-1/2}$ is the positive-definite square root of $(\nabla^2 \psi(x))^{-1}$.
For background on stochastic calculus, the reader may consult sources such as Kallenberg \cite{ka} or
 {\O}ksendal \cite{oksendal}.
The {\it stochastic completeness} of $M$ is equivalent to the existence of a solution $(Y_t)_{t \geq 0}$ to the equation (\ref{eq_342}),
with an initial condition $Y_0 = z$ for a fixed $z \in \RR^n$, that does not explode in finite time. Proposition \ref{prop_957} therefore follows from the next proposition:

\begin{proposition} Let $\psi, V$ and $\rho$ be as in Proposition \ref{prop_957}.
Fix $z \in \RR^n$. Then there exists a unique stochastic process $(Y_t)_{t \geq 0}$, adapted to the filtration
induced by the Brownian motion,  such that
for all $t \geq 0$,
\begin{equation}  Y_{t} = z + \int_0^{t} \sqrt{2} \left( \nabla^2 \psi \left(Y_t \right) \right)^{-1/2} d B_t - \int_0^{t} \nabla \rho(\nabla \psi(Y_t)) dt,
\label{eq_211} \end{equation}
and such that the map $t \mapsto Y_t \ \ (t \geq 0)$ is  almost-surely continuous. \label{prop_1150}
\end{proposition}

\begin{proof} Since $\psi(x)$ tends to $+\infty$ when $x \rightarrow \infty$,
then the convex set $\{ \psi \leq R \} = \{ x \in \RR^n ; \psi(x) \leq R \}$
is compact for any $R \in \RR$.
We use Theorem 21.3 in Kallenberg \cite{ka} and the remark following it.
We deduce that there exists a unique continuous stochastic process $(Y_t)_{t \geq 0}$ and stopping times $T_k = \inf \{ t \geq 0 ; \psi(Y_t) \geq k \}$
such that for any $k > \psi(z), t \geq 0$,
\begin{equation}  Y_{ \min \{t, T_k \} } = z + \int_0^{\min \{t, T_k \}} \sqrt{2} \left( \nabla^2 \psi(Y_t) \right)^{-1/2} d B_t - \int_0^{\min \{t, T_k \}}
\nabla \rho(\nabla \psi(Y_t)) dt.
\label{eq_955} \end{equation}
Denote $T = \sup_k T_k$. We would like to prove that $T = +\infty$ almost-surely.  According to Dynkin's formula and
Lemma \ref{lem_1038}, for any $k > \psi(z)$ and $t \geq 0$,
$$ \EE \psi( Y_{ \min \{t, T_k \} } ) = \psi(z) + \int_0^{\min \{t, T_k \}} (L \psi)(Y_t) dt \leq \psi(z) + 2 A t, $$
where $A$ is the parameter  from Lemma \ref{lem_1038}.
Set $\alpha = -\inf_{x \in \RR^n} \psi(x)$, a finite number in view of (\ref{eq_1205}).
Then $\psi(x) + \alpha$ is non-negative. By Markov-Chebyshev's inequality, for any $t \geq 0$ and $k > \psi(z)$,
$$ \PP(T_k \leq t) = \PP \left( \psi( Y_{ \min \{t, T_k \} } ) \geq k \right) \leq 
\frac{\EE \psi( Y_{ \min \{t, T_k \} } ) + \alpha}{k + \alpha} 
\leq
\frac{2 At + \psi(z) + \alpha}{k + \alpha}. $$
Hence, for any $t \geq 0$,
$$ \PP(T \leq t) \leq \inf_k \PP(T_k \leq t) \leq \liminf_{k \rightarrow \infty} \frac{2 At + \psi(z) + \alpha}{k + \alpha} = 0. $$
Therefore $T = +\infty$ almost surely. We may let $k$ tend to infinity in (\ref{eq_955}) and deduce (\ref{eq_211}).
The uniqueness of the continuous stochastic process $(Y_t)_{t \geq 0}$ that satisfies (\ref{eq_211}) follows from
the uniqueness of the solution to (\ref{eq_955}).
\end{proof}

For $z \in \RR^n$ write $(Y_t^{(z)})_{t \geq 0}$ for the stochastic process from Proposition \ref{prop_1150}
with $Y_0 = z$.
Denote by $\nu$ the probability measure on $\RR^n$ whose density is $e^{-V(x)} dx$.
The lemma below is certainly part of the standard theory of diffusion processes. We were not able to
find a precise reference, hence we provide a proof which relies on the existence of the heat kernel.

\begin{lemma} There exists a smooth function $p_t(x,y) \ (x, y \in \RR^n, t > 0)$ which is symmetric in $x$ and $y$, such that
for any $y \in \RR^n$ and $t > 0$, the random vector
$$ Y_t^{(y)} $$
has density $x \mapsto p_t(x,y)$ with respect to $\nu$. \label{lem_343}
\end{lemma}

\begin{proof} We appeal to Theorem 7.13 and Theorem 7.20 in Grigor'yan \cite{gri}, which deals
with heat kernels on weighted Riemannian manifolds. According to these theorems, there exists a heat kernel, that is, a non-negative function $p_t(x,y) \ (x,y \in \RR^n, t > 0)$ symmetric in $x$ and
$y$ and smooth
jointly in $(t,x,y)$, that satisfies the following two properties:
\begin{enumerate}
\item[(i)] For
any $y \in \RR^n$, the function $u(t,x) = p_t(x,y)$ satisfies
$$\frac{\partial u(t,x)}{\partial t} = L_x u(t,x) \quad \quad \quad \quad (x \in \RR^n, t > 0)
$$
where by $L_x u(t,x)$ we mean that the operator $L$ is acting on the $x$-variables.
\item[(ii)] For any smooth, compactly-supported function $f: \RR^n \rightarrow \RR$ and $x \in \RR^n$,
\begin{equation}
 \int_{\RR^n} p_t(x,y) f(y) d \nu(y) \stackrel{t \rightarrow 0^+}\longrightarrow f(x), \label{eq_1049_}
 \end{equation}
 and the convergence in (\ref{eq_1049_}) is locally uniform in $x \in \RR^n$.
 \end{enumerate}
Theorem 7.13 in Grigor'yan \cite{gri} also guarantees that $\int p_t(x,y) d \nu(x) \leq 1$ for any $y$.
It remains to prove that the random vector $Y_t^{(y)}$ has density $x \mapsto p_t(x,y)$ with respect to $\nu$. Equivalently, we need to show that
for any smooth, compactly-supported function $f: \RR^n \rightarrow \RR$ and $y \in \RR^n, t > 0$,
\begin{equation}
\EE f \left(Y_t^{(y)} \right) = \int_{\RR^n} f(x) p_t(x,y) d \nu(x). \label{eq_306}
\end{equation}
Denote by $v(t,y) \ \ (t > 0, y \in \RR^n)$ the right-hand side of (\ref{eq_306}), a smooth, bounded function.
We also set $v(0, y) = f(y)  \ \ (y \in \RR^n)$ by continuity, according to (ii).
Then the function $v(t,y)$ is continuous and bounded in $(t,y) \in [0, +\infty) \times \RR^n$.
Since $f$ is compactly-supported then we may safely differentiate under the integral sign
with respect to $y$ and $t$,
and obtain
$$ \frac{\partial v(t,y)}{\partial t} = \int_{\RR^n} f(x) \frac{\partial p_t(x,y)}{\partial t} d \nu(y),
\quad L_y v(t,y) =  \int_{\RR^n} f(x)  \left( L_y p_t(x,y) \right)  d \nu(y). $$
From (i) we learn that
\begin{equation} \frac{\partial v(t,y)}{\partial t} = L_y v(t,y) \quad \quad \quad \quad (y \in \RR^n, t > 0).
\label{eq_316}
\end{equation}
Fix $t_0 > 0$ and $y \in \RR^n$. Denote $Z_t = v \left( t_0 - t, Y_t^{(y)} \right)$ for $0 \leq t \leq t_0$.
Then $(Z_t)_{0 \leq t \leq t_0}$ is a continuous stochastic process. From It\^o's formula and (\ref{eq_316}), for $0 \leq t \leq t_0$,
$$  Z_t = Z_0 \, + \, R_t \, + \, \int_0^t \left[ L_y v \left(t_0 - t, Y_t^{(y)} \right)  \, - \,  \frac{\partial v}{\partial t} \left(t_0 - t, Y_t^{(y)} \right) \right] dt
 = Z_0 \, + \, R_t $$
where $(R_t)_{0 \leq t \leq t_0}$ is a local martingale with $R_0 = 0$. Since $v$ is bounded, then $(R_t)_{0 \leq t \leq t_0}$ is in fact a martingale, and
in particular $\EE R_{t_0} = \EE R_0 = 0$. Thus,
$$ \EE f \left(Y_{t_0}^{(y)} \right) = \EE Z_{t_0} = \EE Z_0 = v(t_0, y) = \int_{\RR^n} f(x) p_{t_0}(x,y) d \nu(x), $$
and (\ref{eq_306}) is proven.
\end{proof}

\begin{corollary} Suppose that $Z$ is a random vector in $\RR^n$, distributed according to $\nu$,
independent of the Brownian motion $(B_t)_{t \geq 0}$ used for the construction of $(Y_t^{(z)})_{t \geq 0, z \in \RR^n}$.

\medskip Then, for any $t \geq 0$, the random vector
$ Y_t^{(Z)} $
is also distributed according to $\nu$.
\label{cor_1206}
\end{corollary}

\begin{proof} According to Lemma \ref{lem_343}, for any measurable set $A \subset \RR^n$,
\begin{align*}  \PP \left( Y_t^{(Z)} \in A \right) & = \int_{\RR^n} \PP \left( Y_t^{(z)} \in A \right) d \nu(z)
= \int_{\RR^n} \left( \int_{A} p_t(z, x) d \nu(x) \right) d \nu(z) \\ & =
\int_{A} \left( \int_{\RR^n} p_t(x, z) d \nu(z) \right) d \nu(x) = \nu(A).  \tag*{\qedhere} \end{align*}
\end{proof}

\begin{remark} {\rm
Our choice to use stochastic processes
in this paper is just a matter of personal taste. All of the arguments here
can be easily rephrased in analytic terminology. For instance, the proof
of Proposition
\ref{prop_1150} relies on the fact that $L \psi$ is bounded from above,
similarly to the analytic approach in Grigor'yan \cite[Section 8.4]{gri}.
Another example is the use of local martingales towards the end of Lemma \ref{lem_343}, which may be replaced
by analytic arguments as in \cite[Section 7.4]{gri}.
}
\end{remark}

\section{Bakry-\'Emery technique}
\label{sec_BE}

In this section we prove Theorem \ref{thm2}.
While the viewpoint and ideas of Bakry and \'Emery \cite{BE} are  certainly the main source of inspiration for our analysis,
we are not sure whether the abstract framework in \cite{B, BE} entirely encompasses the subtlety of our specific weighted Riemannian manifold.
For instance, Lemma \ref{lem_1622} below seems related to
the positivity of the
{\it carr\'e du champ} $\Gamma_2$ and to the inequality
$\Gamma_2 \geq \Gamma/2$, rendered as property (ii) in Section 1 above.
In the case $\eps \geq 1/2$, Lemma \ref{lem_1622} actually follows from an application of \cite[Lemma 2.4]{B}
with $f(x) = x^1$ and $\rho =1/2$. Yet, in general, it appears to us advantageous to proceed by analyzing our model for itself, rather
than viewing it as an abstract diffusion semigroup satisfying a curvature-dimension bound.

\medskip
Let $\mu$ be a log-concave probability measure on $\RR^n$
satisfying the regularity assumptions
(\arabic{eq_935}),
whose barycenter lies at the origin. Let $\psi: \RR^n \rightarrow \RR$ be convex and smooth, such
that the transport equation (\ref{eq_1120}) holds true.
In Section \ref{sec_diffusion} we proved that $M_{\mu}^*$ is stochastically complete.
Since $M_{\mu^*}$ is isomorphic to $M_{\mu}$, then  $M_{\mu}$ is also stochastically complete.

\medskip Let us describe in greater detail the diffusion process associated with $M_{\mu} = (K, \nabla^2 \vphi, \mu)$.
Recall that the Legendre transform $\vphi = \psi^*$
is smooth and convex on $K$, and that
$$  \vphi(x) + \psi(\nabla \vphi(x)) = x \cdot \nabla \vphi(x) \quad \quad \quad \quad (x \in K).
$$
We may rephrase (\ref{eq_1120}) in terms of $\vphi = \psi^*$, and using
$(\nabla^2 \vphi(x))^{-1} = \nabla^2 \psi(\nabla \vphi(x))$, we
arrive at the equation
\begin{equation} \det \nabla^2 \vphi(x) = e^{x \cdot \nabla \vphi(x) - \vphi(x) - \rho(x)}
\quad \quad \quad \quad (x \in K). \label{eq_1234}
\end{equation}
The Hessian matrix $\nabla^2 \vphi$ is invertible everywhere, so we write $\left( \nabla^2 \vphi(x) \right)^{-1} = ( \vphi^{ij}(x) )_{i,j=1,\ldots,n}$, and
as before we also use abbreviations such as  $\vphi^{jk}_{i} = \vphi^{j \ell} \vphi^{k m} \vphi_{i \ell m}$. In this section, for a smooth function $u: K \rightarrow \RR$,  denote
\begin{equation}
L u(x) = \vphi^{ij} u_{ij} - x^i u_i \quad \quad \quad \quad \text{for} \ x = (x^1,\ldots,x^n) \in K.
\label{eq_1105} \end{equation}

\begin{lemma}
The operator $L$ from (\ref{eq_1105}) is the Laplacian associated with the weighted Riemannian manifold $M_{\mu}$.
\end{lemma}

\begin{proof} By taking the logarithmic derivative of (\ref{eq_1234}) and arguing as in the proof of Lemma \ref{lem_950},
we obtain that for any $x \in K, i=1,\ldots,n$,
\begin{equation}
\vphi^{i j}_{j} = x^i - \vphi^{i j} \rho_{j}. \label{eq_1301} \end{equation}
Integrating by parts and using (\ref{eq_1301}), we see that for any two smooth functions $u,v: K \rightarrow \RR$
with one of them compactly-supported,
\begin{align*}
\int_K \vphi^{ij} u_i v_j d \mu = -\int_K v (\vphi^{ij} u_{ij} - (\vphi^{ij}_j + \vphi^{ij} \rho_j) u_i) e^{-\rho} = -\int_K v (Lu) d \mu.
\tag*{\qedhere} \end{align*}
\end{proof}

\begin{lemma} Fix $\eps > 0$. For $x \in K$ set $f(x) = \vphi^{11}(x)$. Then,
$$ L \left( f^{\eps} \right) + \eps f^{\eps} \geq 0. $$ \label{lem_1622}
\end{lemma}

\begin{proof} For  $i,j=1,\ldots,n$,
$$ f_i = (\vphi^{11})_i = -\vphi^{1k} \vphi^{1 \ell} \vphi_{ik\ell}, \quad \quad f_{ij} = -\vphi^{11}_{ij}
+ 2 \vphi^{1k}_j \vphi^1_{i k}. $$
Therefore,
\begin{equation}
 L f = \vphi^{ij} f_{ij} - x^i f_i = -\vphi^{11j}_j + 2 \vphi^{1j}_i \vphi^{1i}_{j} + x^j \vphi^{11}_j. \label{eq_1525} \end{equation}
Taking the logarithm of (\ref{eq_1234}) and differentiating with respect to $x^i$ and $x^{\ell}$, we
see that
$$ \vphi^j_{ji \ell} - \vphi_i^{jk} \vphi_{jk\ell} = -\rho_{i \ell} + \vphi_{i \ell} + x^j \vphi_{i \ell j} \quad \quad \quad (i,\ell=1,\ldots,n). $$
Multiplying by $\vphi^{1 i} \vphi^{1 \ell}$ and summing yields
\begin{equation}
 \vphi^{j11}_{j} - \vphi_k^{1j} \vphi^{1k}_{j} = -\vphi^{1 i} \vphi^{1 \ell} \rho_{i \ell} + \vphi^{11} + x^j \vphi^{11}_{j}.
 \label{eq_1535} \end{equation}
 Since $\rho$ is convex then its Hessian matrix is non-negative definite and $\rho_{i \ell} \vphi^{1 i} \vphi^{1 \ell} \geq 0$.
 From (\ref{eq_1525}) and (\ref{eq_1535}),
 \begin{equation} L f = \vphi^{1j}_k \vphi^{1k}_{j} - \vphi^{11} + \rho_{i \ell} \vphi^{1 i} \vphi^{1 \ell} \geq \vphi^{1j}_k \vphi^{1k}_{j} - \vphi^{11} = \vphi^{1j}_k \vphi^{1k}_{j} -f.
 \label{eq_1455} \end{equation}
The chain rule of the Laplacian is $L (\lambda(f)) = \lambda^{\prime}(f) L f +
\lambda^{\prime \prime}(f) \vphi^{ij} f_j f_j$, as may be verified directly. Using the chain rule with $\lambda(t) =t^{\eps}$ we see that
(\ref{eq_1455}) leads to
$$
L \left( f^{\eps} \right) = \eps f^{\eps - 1} L f + \eps (\eps - 1) f^{\eps - 2} \vphi^{11j} \vphi^{11}_j \geq
 \eps f^{\eps - 1} \vphi^{1j}_k \vphi^{1k}_{j} -\eps f^{\eps} + \eps (\eps - 1) f^{\eps - 2} \vphi^{11j} \vphi^{11}_j. $$
That is,
\begin{equation}
L \left( f^{\eps} \right) + \eps f^{\eps} \geq \eps f^{\eps-1} \left[ \vphi^{1j}_k \vphi^{1k}_{j} + (\eps - 1) \frac{\vphi^{11j} \vphi^{11}_j}{\vphi^{11}}
\right] \geq \eps f^{\eps-1} \left[ \vphi^{1j}_k \vphi^{1k}_{j} - \frac{\vphi^{11j} \vphi^{11}_j}{\vphi^{11}}
\right], \label{eq_1502}
\end{equation}
 where we used the fact that $\vphi^{11j} \vphi^{11}_j \geq 0$ in the last passage (or more generally, $\vphi^{ij} h_i h_j \geq 0$ for
 any smooth function $h$). It remains to show that the right-hand side of (\ref{eq_1502}) is non-negative.
 Denote $A = (\vphi^{1j}_k)_{j,k=1,\ldots,n}$. The matrix $B = (\vphi^{1jk})_{j,k=1,\ldots,n}$ is a symmetric matrix, since
$\vphi^{1 jk} = \vphi^{1 \ell} \vphi^{j m} \vphi^{k r} \vphi_{\ell m r}$. We have $A = (\nabla^2 \vphi) B$, and hence
\begin{align*}
 \vphi^{1j}_k \vphi^{1k}_{j} & = Tr(A^2) = Tr \left[ \left( (\nabla^2 \vphi)^{1/2} B (\nabla^2 \vphi)^{1/2} \right)^2 \right] =
 \left \| (\nabla^2 \vphi)^{1/2} B (\nabla^2 \vphi)^{1/2} \right \|_{HS}^2,
\end{align*}
since the matrix $(\nabla^2 \vphi)^{1/2} B (\nabla^2 \vphi)^{1/2}$ is symmetric,  where $\| T \|_{HS}$ stands for the Hilbert-Schmidt norm of the matrix $T$. We will use
 the fact that the Hilbert-Schmidt norm is at least as large as the operator norm, that is,
 $\| T \|_{HS}^2 \geq |T x|^2 / |x|^2$ for any $0 \neq x \in \RR^n$. Setting $e_1=(1,0,\ldots,0)$, we conclude that
 \begin{equation}
 \vphi^{1j}_k \vphi^{1k}_{j} \geq \frac{ \left| (\nabla^2 \vphi)^{1/2} B (\nabla^2 \vphi)^{1/2} (\nabla^2 \vphi)^{-1/2} e_1 \right|^2}{
 \left| (\nabla^2 \vphi)^{-1/2} e_1 \right|^2} = \frac{ \vphi^{11i} \vphi_{ij} \vphi^{11j}}{\vphi^{11}}
= \frac{\vphi^{11}_j \vphi^{11j}}{\vphi^{11}}. \label{eq_1521}
 \end{equation}
 The lemma follows from (\ref{eq_1502}) and (\ref{eq_1521}).
\end{proof}

Let $(B_t)_{t \geq 0}$ be the standard $n$-dimensional Brownian motion.
From the results of Section \ref{sec_diffusion}, the diffusion process whose generator is $L$ from (\ref{eq_1105}) is well-defined. That is,
there exists a unique stochastic process $(X_t^{(z)})_{t \geq 0, z \in K}$, continuous in $t$ and adapted to the filtration
induced by the Brownian motion, such that
for all $t \geq 0$,
\begin{equation}  X^{(z)}_{t} = z + \int_0^{t} \sqrt{2} \left( \nabla^2 \vphi \left(X_t^{(z)} \right) \right)^{-1/2} d B_t - \int_0^{t} X_t^{(z)} dt.
\label{eq_1328} \end{equation}
Our proof of Theorem \ref{thm2} relies on a few lemmas in which the main technical obstacle  is
to prove the integrability of certain local martingales.
\begin{lemma} Fix $z \in K$ and set $X_t = X_t^{(z)} \ (t \geq 0)$.  Then for any $t \geq 0$,
\begin{equation}
 \EE X_t = e^{-t} z,
 \label{eq_1448_} \end{equation}
and for any $\theta \in S^{n-1}$,
\begin{equation}
 e^{2 t} \EE  (X_t \cdot \theta)^2
\geq (z \cdot \theta)^2 +  2 \int_0^t e^{2s} \EE \left[ (\nabla^2 \vphi)^{-1} (X_s) \theta \cdot \theta \right] ds.
\label{eq_1448}
\end{equation}
\label{lem_1450}
\end{lemma}

\begin{proof} From It\^o's formula and (\ref{eq_1328}),
$$
 d(e^t X_t) = e^t d X_t + e^t X_t dt = \sqrt{2} e^t \left( \nabla^2 \vphi(X_t) \right)^{-1/2} d B_t. $$
Therefore $(e^t X_t)_{0 \leq t \leq T}$ is a local martingale, for any fixed number $T > 0$. However, $e^t X_t \in e^T K$ for $0 \leq t \leq T$,
and $K \subset \RR^n$ is a bounded set. Therefore $(e^t X_t)_{0 \leq t \leq T}$ is a bounded process, and hence it is a martingale. We conclude that
$$ \EE e^t X_t = \EE e^0 X_0 = z \quad \quad \quad \quad (t \geq 0), $$
and (\ref{eq_1448_}) is proven. It remains to prove (\ref{eq_1448}).
Without loss of generality we may assume that $\theta = e_1 = (1,0,\ldots,0)$.
Denote $Z_t = e^{2t} (X_t \cdot e_1)^2$. According to (\ref{eq_1328}) and It\^o's formula, for any $t \geq 0$,
\begin{equation} Z_t = (z \cdot e_1)^2 + M_t  + \int_0^t \left( 2 e^{2s} \vphi^{11}(X_s) \right) ds, \label{eq_956} \end{equation}
where $(M_t)_{t \geq 0}$ is a local martingale with $M_0 = 0$. Since $\vphi^{11}$ is positive, then for any $t \geq 0$,
\begin{equation} Z_t - (z \cdot e_1)^2 \geq M_t. \label{eq_1223} \end{equation}
 Since $K$ is bounded, then $(Z_t)_{0 \leq t \leq T}$ is a bounded process,
for any number $T > 0$. According to (\ref{eq_1223}), the local martingale $(M_t)_{0 \leq t \leq T}$
is bounded from above, and by Fatou's lemma it is a sub-martingale.
In particular $\EE M_t \geq \EE M_0 = 0$ for any $t$.
From (\ref{eq_956}),
$$ \EE Z_t \geq (z \cdot e_1)^2 + 2 \EE \int_0^t e^{2s} \vphi^{11}(X_s)  ds
\quad \quad \quad \quad (t \geq 0).
$$
Since $\EE Z_t < +\infty$ and $\vphi^{11}$ is positive,  we may use Fubini's theorem to conclude that for any $t \geq 0$,
\begin{align*}
 \EE Z_t \geq (z \cdot e_1)^2 +  2 \int_0^t e^{2s} \EE \vphi^{11}(X_s) ds. \tag*{\qedhere} \end{align*}
\end{proof}

\begin{remark}{\rm Once Theorem \ref{thm2} is established, we can prove that equality holds in (\ref{eq_1448}).
Indeed, it follows from Theorem \ref{thm2} and (\ref{eq_956}) that $(M_t)_{0 \leq t \leq T}$ is a bounded process and
hence a martingale.
}\end{remark}

\begin{lemma}
Assume that
the convex body $K$ has a smooth boundary
and that its Gauss curvature is positive everywhere. Assume also that there exists $\eps_0 > 0$ with
\begin{equation}  \nabla^2 \rho(x)  \geq \eps_0 \cdot Id \quad \quad \quad \quad (x \in K) \label{eq_1229} \end{equation}
in the sense of symmetric matrices. Fix $z \in K$ and set $X_t = X_t^{(z)} \ (t \geq 0)$. Denote $f(x) = \vphi^{11}(x)$ for $x \in K$. Then, for any $t, \eps > 0$,
\begin{equation}
f(z) \leq e^t \left( \EE f^{\eps}(X_t) \right)^{1/\eps}. \label{eq_1233}
  \end{equation}
\label{lem_1229}
\end{lemma}

\begin{proof} Our assumptions enable the application of Proposition \ref{prop_1101}. According to the conclusion of Proposition \ref{prop_1101},
there exists $M > 0$ such that
$$ \nabla^2 \psi(y) \leq M \cdot Id \quad \quad \quad \quad (y \in \RR^n). $$
Since $(\nabla^2 \vphi)^{-1}(x) = \nabla^2 \psi(\nabla \vphi(x))$, then,
\begin{equation}
 f(x) = \vphi^{11}(x) \leq M \quad \quad \quad \quad (x \in K). \label{eq_1231}
 \end{equation}
 From It\^o's formula and (\ref{eq_1328}),
\begin{equation}
 e^{\eps t} f^{\eps}(X_t) = f^{\eps}(z) + M_t + \int_0^t e^{\eps s} \left[ (L f^{\eps})(X_s) + \eps f^{\eps}(X_s) \right] ds, \label{eq_1015_}
 \end{equation}
where $M_t$ is a local martingale with $M_0 = 0$. According to (\ref{eq_1015_}) and Lemma \ref{lem_1622}, for any $t \geq 0$,
\begin{equation} e^{\eps t} f^{\eps}(X_t) \geq f^{\eps}(z) + M_t. \label{eq_1626} \end{equation}
We may now use (\ref{eq_1231}) and (\ref{eq_1626}) in order to conclude that the local martingale $(M_t)_{0 \leq t \leq T}$ is bounded from above,
for any number $T > 0$.
Hence it is a sub-martingale, and $\EE M_t \geq \EE M_0 = 0$ for any $t \geq 0$.
Now (\ref{eq_1233}) follows by taking
the expectation of (\ref{eq_1626}).
 \end{proof}

\begin{remark}{\rm
We will only use (\ref{eq_1233}) for $\eps = 1$, even though the statement for a small $\eps$ is much stronger. In the limit
where $\eps$ tends to zero, it is not too difficult to prove that the right-hand side of (\ref{eq_1233}) approaches $\exp(t + \EE \log f(X_t))$.
}\end{remark}

The covariance matrix of a square-integrable random vector $Z = (Z_1,\ldots,Z_n) \in \RR^n$ is defined to be
$$ Cov(Z) = \left( \EE Z_i Z_j - \EE Z_i \cdot \EE Z_j \right)_{i,j=1,\ldots,n}. $$

\begin{corollary}Assume that
the convex body $K$ has a smooth boundary
and that its Gauss curvature is positive everywhere. Assume also that there exists $\eps_0 > 0$ with
\begin{equation}  \nabla^2 \rho(x)  \geq \eps_0 \cdot Id \quad \quad \quad \quad (x \in K).
\label{eq_1141_} \end{equation}
Then for any $z \in K$ and $t > 0$,
$$ (\nabla^2 \vphi)^{-1}(z) \leq \frac{e^{2t}}{2(e^t-1)} \cdot Cov \left(X_t^{(z)} \right)
 $$
 in the sense of symmetric matrices.
 \label{cor_1501}
\end{corollary}

\begin{proof} Fix $z \in K, t > 0$ and $\theta \in S^{n-1}$. We need to prove that
\begin{equation}
\left( \nabla^2 \vphi(z) \right)^{-1} \theta \cdot \theta \leq \frac{e^{2t}}{2(e^t-1)} Var(X_t^{(z)} \cdot \theta). \label{eq_1435}
\end{equation}
Without loss of generality we may assume that $\theta = e_1 = (1,0,\ldots,0)$.
We use Lemma \ref{lem_1450} and also Lemma \ref{lem_1229} with $\eps = 1$, and obtain
$$ e^{2 t} \EE  (X_t^{(z)} \cdot e_1)^2
\geq (z \cdot e_1)^2 +  2 \int_0^t e^{2s} \EE \vphi^{11} (X_s^{(z)}) ds
\geq (z \cdot e_1)^2 +  2 \vphi^{11}(z) \int_0^t e^{s} ds. $$
Recall that $\EE X_t^{(z)} = e^{-t} z$, according to Lemma \ref{lem_1450}. Consequently,
\begin{align*}
\vphi^{11}(z) \leq \frac{e^{2t}}{2(e^t-1)} \left( \EE  (X_t^{(z)} \cdot e_1)^2 - (e^{-t} z \cdot e_1)^2 \right) =
\frac{e^{2t}}{2(e^t-1)} Var(X_t^{(z)} \cdot e_1),
\end{align*}
and (\ref{eq_1435}) is proven for $\theta = e_1$.
\end{proof}

\begin{proof}[Proof of Theorem \ref{thm2}] Assume first that
the convex body $K$ has a smooth boundary, that
its Gauss curvature is positive everywhere, and that there exists $\eps_0$ for which (\ref{eq_1141_}) holds true.
We apply Corollary \ref{cor_1501} with $t = \log 2$, and conclude that for any $z \in K$,
$$ Tr\left[
(\nabla^2 \vphi)^{-1}(z) \right] \leq 2 Tr \left[ Cov(X_t^{(z)}) \right] \leq 2 \EE \left|X_t^{(z)} \right|^2 \leq 2 R^2(K) $$
as $X_t^{(z)} \in K$ almost surely. Therefore, for any $x \in \RR^n$, setting $z = \nabla \psi(x)$ we have
\begin{equation}
\Delta \psi(x) = Tr\left[
\nabla^2 \psi(x) \right] = Tr\left[
(\nabla^2 \vphi)^{-1}(z) \right] \leq 2 R^2(K). \label{eq_1504}
\end{equation}
It still remains to eliminate the extra strict-convexity assumptions.
To that end, we select a sequence of smooth convex bodies $K_{\ell} \subset \RR^n$, each with a
positive Gauss curvature, that converge in the Hausdorff metric to $K$. We then consider
a sequence of log-concave probability measures $\mu_{\ell}$ with barycenter at the origin that converge weakly
to $\mu$, such that
$\mu_{\ell}$ is supported on $K_{\ell}$ and such that the smooth density of $\mu_{\ell}$ satisfies (\ref{eq_1141_}) with,
say, $\eps_0 = 1/\ell$. We also assume that $\mu_{\ell}$ and $K_{\ell}$  satisfy the regularity conditions (\arabic{eq_935}).

\medskip It is not very difficult to construct the $\mu_{\ell}$'s: For instance, convolve $\mu$ with a
tiny Gaussian (this preserves log-concavity), multiply the density by $\exp(-|x|^2 / \ell)$, truncate with $K_{\ell}$ and translate
a little so that the barycenter would lie at the origin. This way we obtain a sequence of smooth, convex functions $\psi_{\ell}: \RR^n \rightarrow \RR$
such that $\mu_{\ell}$ is the moment measure of $\psi_{\ell}$. We may translate, and assume that $\psi$ and each of the $\psi_{\ell}'s$ are {\it centered},
in the terminology of Section \ref{sec_continuity}. According to (\ref{eq_1504}), we know that
\begin{equation}
\Delta \psi_{\ell}(x) \leq 2 R^2(K_{\ell}) \quad \quad \quad \quad (x \in \RR^n, \ell \geq 1). \label{eq_1519}
\end{equation}
Furthermore, $\mu_{\ell} \longrightarrow \mu$ weakly, and by Proposition \ref{prop_1130}, also $\psi_{\ell} \longrightarrow \psi$
pointwise in $\RR^n$. Since $\psi_{\ell}$ and $\psi$ are smooth, then \cite[Theorem 24.5]{roc} implies that
$$ \nabla \psi_{\ell}(x) \stackrel{\ell \rightarrow \infty}\longrightarrow \nabla \psi(x)
 \quad \quad \quad \quad (x \in \RR^n).
 $$
The function $\psi_{\ell}$
is $R(K_{\ell})$-Lipschitz, and $R(K_{\ell}) \longrightarrow
R(K)$. Hence $\sup_{\ell, x} |\nabla \psi_{\ell}(x)|$ is finite. By the bounded convergence theorem,
for any $x_0 \in \RR^n$ and $\eps > 0$,
\begin{equation}
 \int_{B(x_0, \eps)} \Delta \psi_{\ell} = \int_{\partial B(x_0, \eps)} \nabla \psi_{\ell} \cdot N
\stackrel{\ell \rightarrow \infty}\longrightarrow
\int_{\partial B(x_0, \eps)} \nabla \psi \cdot N = \int_{B(x_0, \eps)} \Delta \psi, \label{eq_859} \end{equation}
where $N$ is the outer unit normal. From (\ref{eq_1519}) and (\ref{eq_859}) we conclude that for any $x_0 \in \RR^n$ and $\eps > 0$,
$$ \int_{B(x_0, \eps)} \Delta \psi \leq  Vol_n( B(x_0, \eps) ) \cdot \limsup_{\ell \rightarrow \infty} 2 R^2(K_{\ell}) =
2 Vol_n( B(x_0, \eps) )  R^2(K), $$
where $Vol_n$ is the Lebesgue measure in $\RR^n$. Since $\psi$ is smooth, then we may let $\eps$ tend to zero and conclude that $\Delta \psi(x_0) \leq 2 R^2(K)$,
for any $x_0 \in \RR^n$.
\end{proof}

Posteriori, we may strengthen Corollary \ref{cor_1501} and eliminate the strict-convexity
assumptions. These assumptions were used only  in the proof of Lemma \ref{lem_1229},
to deduce the existence of some number $M > 0$ for which $\nabla^2 \psi(x) \leq M \cdot Id$, for all $x \in \RR^n$.
Theorem \ref{thm2} provides such a number $M = 2 R^2(K)$, without any strict-convexity assumptions
on $\rho$ or $K$. We may therefore upgrade Corollary \ref{cor_1501}, and conclude that

\begin{corollary} Whenever $\mu$ is a log-concave probability measure with barycenter at the origin,
satisfying the regularity conditions (\arabic{eq_935}), we have
$$
 (\nabla^2 \vphi)^{-1}(z) \leq \frac{e^{2t}}{2(e^t-1)} \cdot Cov \left( X_t^{(z)} \right)
 $$ in the sense of symmetric matrices,  for any $z \in K$ and $t > 0$.
 \label{cor_905}
\end{corollary}

\section{The Brascamp-Lieb inequality as a Poincar\'e inequality}

We retain 
the assumptions and notation of the previous section. That is,
$\mu$ is a log-concave probability measure on $\RR^n$, with barycenter
at the origin, that satisfies the regularity assumptions
(\arabic{eq_935}). The measure $\mu$ is the moment-measure of the smooth and convex function $\psi: \RR^n \rightarrow \RR$.
Equation (\ref{eq_1120}) holds true, and we denote $\vphi = \psi^*$.
According to the Brascamp-Lieb inequality \cite{BL}, for any smooth function $u: \RR^n \rightarrow \RR$
such that $u e^{-\psi}$ is integrable,
\begin{equation} \int_{\RR^n} u e^{-\psi} = 0 \quad \quad \Longrightarrow \quad \quad \int_{\RR^n} u^2 e^{-\psi}
\leq \int_{\RR^n} \left[ (\nabla^2 \psi)^{-1} \nabla u \cdot \nabla u \right] e^{-\psi}. \label{eq_1442}
\end{equation}
Equality in (\ref{eq_1442}) holds when $u(x) = \nabla \psi(x) \cdot \theta$ for some $\theta \in \RR^n$.
Note that (\ref{eq_1442}) is precisely the Poincar\'e inequality with the best constant of the weighted Riemannian manifold
$M_{\mu}^*$. By using the isomorphism between $M_{\mu}$ and $M_{\mu}^*$, we translate (\ref{eq_1442}) as follows:
For any smooth function $f: K \rightarrow \RR$ which is $\mu$-integrable,
\begin{equation} Var_{\mu}(f) \leq \int_{K} \left( \vphi^{ij} f_i f_j \right) d \mu, \label{eq_1443_}
\end{equation}
where $Var_{\mu}(f) = \int f^2 d \mu- (\int f d \mu)^2$. Equality in (\ref{eq_1443_}) holds when $f(x) = A + x \cdot \theta$
for some $\theta \in \RR^n$ and $A \in \RR$.
This is in accordance with the fact that linear functions are eigenfunctions, i.e.,
$$ L x^i = -x^i \quad \quad \quad \quad (i=1,\ldots,n) $$
where $L u = \vphi^{ij} u_{ij} - x^i u_i$ is the Laplacian of the
weighted Riemannian manifold $M_{\mu}$. In fact, (\ref{eq_1443_}) means that the
spectrum of the (Friedrich extension of the) operator $L$ cannot intersect the interval $(-1,0)$,
and that the restriction of $-L$ to the subspace of mean-zero functions is at least the identity operator, in the sense
of symmetric operators.

\medskip Theorem \ref{thm2} states that $\Delta \Psi(x) \leq 2 R^2(K)$ everywhere in $\RR^n$. A weak conclusion
is that $\nabla^2 \psi(x) \leq 2 R^2(K) \cdot Id$, or rather, that $(\nabla^2 \vphi(x))^{-1} \leq 2 R^2(K) \cdot Id$.
By substituting this information into (\ref{eq_1443_}), we see that for any smooth function $f \in L^1(\mu)$,
\begin{equation}
 Var_{\mu}(f) \leq 2 R^2(K) \int_{K} |\nabla f|^2 d\mu. \label{eq_1119}
 \end{equation}
This completes the proof of Corollary \ref{cor_920}.
See \cite{K_moment} for more Poincar\'e-type inequalities that are obtained by imposing a
Riemannian structure on the convex body $K$.
The Kannan-Lovas\'z-Simonovits conjecture speculates that $R^2(K)$ in (\ref{eq_1119}) may be
replaced by a universal constant times $\| Cov(\mu) \|$, where $Cov(\mu)$ is the covariance matrix of the random vector
that is distributed according to $\mu$, and $\| \cdot \|$ is the operator norm.

\medskip A potential way to make progress towards the Kannan-Lovas\'z-Simonovits conjecture is to try to
bound the matrices $(\nabla^2 \vphi)^{-1}(x) \ (x \in K)$ in terms of $Cov(\mu)$.
 The following proposition provides a modest step in this direction:

\begin{proposition} Fix $\theta \in S^{n-1}$ and denote
$$ V = \int_{\RR^n} (x \cdot \theta)^2 d \mu(x). $$
Then, for any $p \geq 1$,
$$ \left( \int_K  \left| \frac{(\nabla^2 \vphi)^{-1} \theta \cdot \theta}{V} \right|^p  d \mu \right)^{1/p} \leq 4 p^2. $$
\end{proposition}

\begin{proof} Without loss of generality, assume that $\theta = e_1 = (1,0,\ldots,0)$. According to Corollary \ref{cor_905},
for any $z \in K$ and $t > 0$,
\begin{equation} \vphi^{11}(z) \leq \frac{e^{2t}}{2 (e^t - 1)} Var \left( X_t^{(z)} \cdot e_1 \right)
\leq \frac{e^{2t}}{2 (e^t - 1)} \EE \left( X_t^{(z)} \cdot e_1 \right)^2.
\label{eq_1132} \end{equation}
Let $Z$ be a random vector that is distributed according to $\mu$, independent of the Brownian motion used
in the construction of the process $(X_t^{(z)})_{t \geq 0, z \in K}$. It follows from Corollary \ref{cor_1206} that for any fixed $t \geq 0$
the random vector $X_t^{(Z)}$ is also distributed according to $\mu$. By setting $t = \log 2$
in (\ref{eq_1132}) and applying H\"older's inequality, we see that for any $p \geq 1$,
\begin{equation}
 \EE \left| \vphi^{11}(Z) \right|^p \leq 2^p \EE \left| X_t^{(Z)} \cdot e_1 \right|^{2p} =
2^p \EE \left| Z \cdot e_1 \right|^{2p}. \label{eq_1307}
\end{equation}
The random vector $Z$ has a log-concave density. According to the Berwald inequality \cite{ber, bor},
\begin{equation} \left( \EE \left| Z \cdot e_1 \right|^{2p} \right)^{1/(2p)} \leq \frac{\Gamma(2p +1 )^{1/(2p)}}{\Gamma(3)^{1/2}}
\sqrt{ \EE \left| Z \cdot e_1 \right|^{2} } \leq \frac{2p}{\sqrt{2}} \sqrt{V}.
\label{eq_1308}
\end{equation}
(The Berwald inequality is formulated in \cite{ber, bor} for the uniform measure on a convex body, but it is
well-known that is applies for all log-concave probability measures. For instance, one may deduce the log-concave version
from the convex-body version by using a marginal argument as in \cite{kl}).
The proposition follows from (\ref{eq_1307}) and (\ref{eq_1308}).
\end{proof}

There are several heuristic arguments that indicate much better bounds for $(\nabla^2 \vphi)^{-1}$ than the
ones proven in this paper. These better bounds, in turn, could lead to interesting Poincar\'e-type inequalities
for log-concave measures. At the moment, it is not clear to us how to rigorously justify these heuristic arguments,
partly because of the annoying fact that our Riemannian structure on $K$ is not geodesically complete.
Still, hoping that this work will be continued, we decided to add the Roman numeral I at the end of its title.

{
}


\end{document}